\documentclass[12pt, leqno]{article}
\usepackage{amsmath,amsthm}
\usepackage{amssymb,latexsym}
\usepackage{enumerate}

\theoremstyle{plain}
\newtheorem{thm}{Theorem}[section]
\newtheorem{lm}[thm]{Lemma}
\newtheorem{cor}[thm]{Corollary}
\newtheorem{prop}[thm]{Proposition}%[\theoreminside]
\newtheoremstyle{dotless}{}{}{\itshape}{}{\bfseries}{}{ }{}
\theoremstyle{dotless}
\newtheorem{thma}{Theorem}

\theoremstyle{definition}
\newtheorem{df}[thm]{Definition}

\newtheorem{rem}[thm]{Remark}

\numberwithin{equation}{section}

\usepackage{hyperref}
\usepackage[mathscr]{eucal}

\newcommand\st{\,|\,}

\newcommand{\ns}[1]{\mathbb{#1}}

\newcommand{\Z}{\ns{Z}}
\newcommand{\Q}{\ns{Q}}
\newcommand{\R}{\ns{R}}
\newcommand{\C}{\ns{C}}

\newcommand{\GL}{\mathrm{GL}}
\newcommand{\I}{\mathrm{I}}

\newcommand{\Aut}{\mathrm{Aut}}
\newcommand{\Img}{\mathrm{Img}}

\newcommand{\diag}{\mathrm{diag}}
\begin{document}

\baselineskip=17pt

\author{R. Lutowski, A. Szczepa\'nski\\
Institute of Mathematics\\
University of Gda\'nsk\\
ul. Wita Stwosza 57\\
80-952 Gda\'nsk\\
Poland\\
E-mail: \texttt{rlutowsk@mat.ug.edu.pl}, \texttt{matas@univ.gda.pl}}

\title{Holonomy groups of flat manifolds with $R_\infty$ property}

\date{}

\maketitle

\renewcommand{\thefootnote}{}
\renewcommand{\thefootnote}{\arabic{footnote}}
\setcounter{footnote}{0}

\begin{abstract}
Let $M$ be a flat manifold. We say that $M$ has $R_\infty$ property if
the~Reidemeister number $R(f) = \infty$ for every homeomorphism $f \colon M \to M.$  
In this paper,
we investigate a~relation between the holonomy representation $\rho$ of a~flat manifold $M$ and the $R_\infty$ property. 
In the case when the holonomy group of $M$ is solvable we show that if $\rho$ has a~unique 
$\mathbb{R}$-irreducible subrepresentation of odd degree then $M$ has $R_\infty$ property.
The result is related to conjecture 4.8 from \cite{DRP}.
\end{abstract}

\section{Introduction}
Let $f \colon M^n \to M^n$ be a continuous map on a closed $n$-dimensional manifold $M^n.$ From a point of view of the fixed
point theory the following three numbers have a particular meaning: the Lefschetz number $L(f)$, the Nielsen number $N(f)$ and the Reidemeister number $R(f)$. 
If $n\geq 3$, the Nielsen number
$N(f)$ is a sharp lower bound on the number of fixed points of any element in the homotopy class of $f$. However in general $N(f)$ is
difficult to calculate. In 1963, B. Jiang identified a large class of spaces for which
$$
N(f) = \left\{\begin{array}{ll}
0&\text{if } L(f) = 0,\\
R(f) &\text{if } L(f)\neq 0,
\end{array}\right.
$$
for all continuous maps $f\colon M^n\to M^n$.

In the light of the above relation,
since the Nielsen number is always finite, the finiteness of the Reidemeister number is important.
This was one of motivations to introduce
\begin{df}
A manifold $M$ has the \emph{$R_{\infty}$ property} if $R(f) =\infty$ for every homeomorphism
$f\colon M^n\to M^n$.
\end{df}

The Reidemeister number can be defined at the level of the fundamental group $\Gamma = \pi_1(M^n).$
Recall that any continuous map $f\colon M^n\to M^n$ induces a morphism $f_{\sharp}\colon\Gamma\to\Gamma$. We say that two elements
$\alpha,\beta\in\Gamma$ are $f_{\sharp}$-conjugated if there exists $\gamma\in\Gamma$ such that
$\beta = \gamma\alpha f_{\sharp}(\gamma)^{-1}$. The $f_{\sharp}$-conjugacy class
$\{\gamma\alpha f_{\sharp}(\gamma)^{-1}\mid \gamma\in\Gamma\}$ of $\alpha$ is called a Reidemeister
class of $f.$ The number of Reidemeister classes is called the Reidemeister number $R(f)$ of $f.$
%A manifold $M^n$ has the $R_{\infty}$ property if $R(f) = \infty$ for every homeomorphism $f\colon M^n\to M^n,$
%see \cite{DRP}.
It is evident that we can also define the above number $R(\Phi)$ for a countable discrete group
$E$ and its automorphism $\Phi.$
We say that a group $E$ has $R_{\infty}$
property if $R(\Phi) = \infty$ for any automorphism $\Phi.$ Moreover,
the family of groups with the $R_{\infty}$ property includes: non-elementary Gromov-hyperbolic groups,
Baumslag-Solitar groups $BS(m,n) = \langle a,b\mid ba^mb^{-1} = a^{n}\rangle$ except for $BS(1,1)$,
lamplighter groups $\Z_n \wr \Z$ if and only if $2|n$ or $3|n$, the Thompson group $F$ and symplectic groups $\mathrm{Sp}(2n,\Z), n \in \Z_+$.
See \cite{felsz} and \cite{TW} for the more comprehensive list and the history  of the  $R_{\infty}$-groups and the complete
bibliography.

% \vskip 1mm
\noindent
Let $M^n$ be a closed Riemannian manifold of dimension $n.$
We shall call $M^n$ flat if, at any point, the sectional curvature is equal to zero. Equivalently, $M^n$ is isometric
to the orbit space $\R^n/\Gamma,$ where $\Gamma$ is a discrete, torsion-free and co-compact subgroup of $O(n)\ltimes\R^n$ = Isom($\R^n$).
From the Bieberbach theorem (see \cite{charlap}, \cite{wolf}) $\Gamma$ defines a short exact sequence of groups
\begin{equation}\label{crysb}
0\rightarrow \Z^n\rightarrow\Gamma\stackrel{p}\rightarrow G\rightarrow 0,
\end{equation}
where $G$ is a finite group.
$\Gamma$ is called a Bieberbach group and $G$ its holonomy group.
%In this paper we shall consider the case of Bieberbach groups.
We can define a holonomy representation $\rho\colon G\to \GL(n,\Z)$ by the formula:
\begin{equation}\label{holonomyrep}
\forall g\in G,\rho(g)(e_i) = \tilde{g}e_i(\tilde{g})^{-1},
\end{equation}
where $e_i\in\Gamma$ are generators of the free abelian group $\Z^n$ for $i=1,2,...,n,$ and $\tilde{g}\in\Gamma$
such that $p(\tilde{g})=g.$

\vskip 1em
In this article we
describe relations between $R_{\infty}$ property of the flat manifold $M^n$ 
(Bieberbach group $\Gamma$) and a structure of its  holonomy representation.
The connections between geometric properties of $M^n$ and algebraic properties of $\rho$ were already considered
in different cases.
For example, Out($\Gamma$) is finite if and only if
the holonomy representation is $\Q$-multiplicity free and any $\Q$-irreducible component
of the holonomy representation is $\R$-irreducible, see \cite{S}. A similar equivalence says 
that an Anosov diffeomorphism $f\colon M^n\to M^n$ exists if and only if 
any $\Q$-irreducible component of a holonomy representation that occurs with multiplicity one is
reducible over $\R,$ see \cite{P}.  
We want to define conditions of this kind for the holonomy representation of 
a flat manifold with $R_{\infty}$
property. We already know that, in this way, the complete characterization is not possible.
There are examples \cite[Th.5.9]{DRP} of flat manifolds $M_1, M_2$ with the same 
holonomy representation such that $M_1$ has $R_{\infty}$ property and $M_2$ has not. 
In \cite[Corollary 4.4]{DRP} it is proved that
if there exists na Anosov diffeomorphism $f\colon M^n\to M^n$ then $R(f)$ is finite and $M^n$ does not have the $R_{\infty}$
property. 
Moreover there exists $M$ such that its holonomy representation has a $\Q$-irreducible component
which is irreducible over $\R$ and occurs with multiplicity one and $M$ does not have the $R_{\infty}$
property, \cite[Example 4.6]{DRP}.
Nevertheless in \cite[Th. 4.7]{DRP} the following is proved:

\begin{thm}[{\cite[Th. 4.7]{DRP}}]
\label{classical} 
Let $M$ be a flat manifold with a holonomy representation $\rho\colon G\to \GL(n,\Z)$ and
let $\rho'\colon G\to \GL(n',\Z)$ be a $\Q$-irreducible $\Q$-subrepresentation
of $\rho$ such that
$\rho'(G)$ is not $\Q$-conjugated to $\tilde{\rho}(G)$ for any other
$\Q$-subrepresentation $\tilde{\rho}$ of $\rho.$ Suppose moreover
that for every $D'\in N_{\GL(n',\Z)}(\rho'(G)),$ there exists $A\in G$ such that
$\rho'(A)D'$ has eigenvalue $1.$ Then $M$ has the $R_{\infty}$ property.
\end{thm}

\begin{rem}
\label{inclusion}
If we assume that 
\begin{equation}\label{simple}
N_{\GL(n',\Q)}(\rho'(G))/C_{\GL(n',\Q)}(\rho'(G)) \cong Aut(G),
\end{equation}
 then
the above requirement that $\rho'(G)$ is not $\Q$-conjugated to $\tilde{\rho}(G)$
is equivalent to the condition that $\rho'\subset\rho$ has multiplicity one.
For example, if we take the diagonal representation $\rho:(\Z_2)^{2n}\to \mathrm{SL}(2n+1,\Z)$
of the elementary abelian $2$-group, then the above equation (\ref{simple}) is not satisfied
for any $\Q$-irreducible subrepresentation of $\rho.$ 
\end{rem}

We shall prove:
\begin{thm} 
Let $M$ be a flat manifold with the holonomy representation $\rho\colon G\to\GL(n,\Z)$ and
let $G$ be a solvable group
and $\rho'\colon G\to \GL(n',\Z)$ be a $\Q$-irreducible $\Q$-subrepresentation
of $\rho$ of odd dimension. If $\rho'(G)$ is not $\Q$-conjugated to $\tilde{\rho}(G),$ for any other
$\Q$-subrepresentation $\tilde{\rho}$ of $\rho$  
then $M$ has the $R_{\infty}$ property.
\end{thm}
If we restrict our consideration to the class of finite groups
which satisfy the condition (\ref{simple}) we have

\begin{thm}
Let $M$ be a flat manifold with the holonomy representation $\rho\colon G\to\GL(n,\Z)$ and
let $G$ be a solvable group
and $\rho'\colon G\to \GL(n',\Z)$ be a $\Q$-irreducible $\Q$-subrepresentation of $\rho$
of mulitiplicity one and odd dimension
which satisfies the condition (\ref{simple})
then $M$ has the $R_{\infty}$ property.
\end{thm}

\noindent
The above result is a corollary from \cite[Th. 5.4.4]{Rob}, the Theorem~\ref{classical} and the following
theorem:

\begin{thma}
Let $G$ be a finite group with a non-trivial normal abelian subgroup $A$ and let 
$\rho\colon G\to \GL(n,\Z)$ be a faithful $\R$-irreducible representation. 
Suppose $n$ is odd. Then for every $D\in N_{\GL(n,\Z)}(\rho(G)),$ 
there exists $g\in G$ such that $\rho(g) D$ has eigenvalue $1$.
\end{thma}
\addtocounter{thma}{-1}

The main idea used in the proof of the main result is
an application of the Clifford's theorem \cite[Theorem 49.2]{CuRe62},
which deals with a relation between irreducible $kG$-modules and $kH$-modules,
where $H$ is a normal subgroup of a finite group $G,$ and $k$ is an arbitrary field.

\begin{rem}
Conjecture 4.8 in \cite{DRP} says that the above Theorem \ref{thma:a} is
true for any finite group.
We do not know whether it holds in general.  
\end{rem}

% \vskip 5mm
\noindent
{\bf Acknowledgement.} 
We would like to thanks G. Hiss for a helpful conversation and particularly for calling our attention to the Clifford's theorem.

\section{Proof of Theorem \ref{thma:a}}

\begin{thm}
Let $G$ be a~finite group and $n$ be an odd integer. Let $\rho \colon G \to \GL(n,\Z)$ be a~faithful representation of $G$
which is irreducible over $\R$. Then $\rho$ is irreducible over $\C$.
\end{thm}

\begin{proof}
Assume that $\rho$ is reducible over $\C$ and let $\tau$ be any $\C$-irreducible subrepresentation of $\rho$. 
By \cite[Theorem 2]{Re61}, the representation $\rho$ is uniquely determined by $\tau$ and, if $\chi$ is the character of $\tau$, then the character of $\rho$ is given by
\[
\chi + \overline{\chi}.
\]
Hence $\rho$ is of even degree. This proves the theorem.
\end{proof}

For the rest of this section we assume that $\rho \colon G \to \GL(n,\Z)$ is an absolutely irreducible representation of $G$, where $n$ is an odd integer. 

\begin{prop}
If $A$ is a normal abelian subgroup of $G$, then $A$ is an elementary abelian $2$-group. 
\end{prop}

\begin{proof}
Let $\tau$ be an $\R$-irreducible subrepresentation of $\rho_{|A}$. 
By Clifford's theorem \cite[Theorem 49.2]{CuRe62}, all $\R$-subrepresentations of $\rho_{|A}$ are conjugates of an
$\R$-irreducible subrepresentation $\tau$, i.e. there exist $g_1 = 1, g_2, \ldots, g_l \in G$ such that
\begin{equation}
\label{eq:decomposition}
\rho_{|A} = \tau^{(g_1)} \oplus \ldots \oplus \tau^{(g_l)}, 
\end{equation}
where
\[
\forall_{1 \leq i \leq l} \forall_{g \in G}\hskip 2mm \tau^{(g_i)}(g) = \tau(g_i^{-1} g g_i).
\]

Let $a \in A$ be an element of order greater than $2.$ Since $\rho$ is faithful, 
there exists $1 \leq i \leq l$ such that $\tau^{(g_i)}(a)$ is a~real 
matrix of order at least $3.$  
Hence $\deg(\tau^{(g_i)}) = \deg(\tau) = 2$ and $n = \deg(\rho) = \deg(\rho_{|A}) = l \deg(\tau) = 2l$ is an even integer. This contradiction finishes the proof.
\end{proof}

Since $A$ is an elementary abelian 2-group, the decomposition \eqref{eq:decomposition} may be realized over the rationals. By \cite[Theorem 49.7]{CuRe62} we may assume that 
\begin{equation}
\label{eq:decomposition2}
\rho_{|A} = e\tau^{(g_1)} \oplus \ldots \oplus e\tau^{(g_k)},
\end{equation}
i.e. one-dimensional representations $\tau^{(g_1)}, \ldots, \tau^{(g_k)}$ 
occur with the same multiplicity $e=n/k.$ Let $\rho_i := e\tau^{(g_i)}$, for $i=1,\ldots,k$. 
By a suitable choice of basis of $\Q^n$ we may assume that for every $a \in A, \rho(a)$ is a~diagonal matrix such that
\begin{equation}
\label{eq:images2}
\forall_{1 \leq i \leq k} \Img(\rho_k) = \langle -\I \rangle,
\end{equation}
where $\I$ is the identity matrix of degree $e$.

Since $A \lhd G$ and $\rho$ is faithful, we have
\[
\rho(A) \lhd \rho(G) \subset N_{\GL(n,\Q)}(\rho(A)) = \{m \in \GL(n,\Q) \st m^{-1} \rho(A) m = \rho(A) \}.
\]

In the next two subsections we will focus on the above normalizer.

\subsection{Centralizer}

In the beginning we describe the centralizer
\[
C_{\GL(n,\Q)}(\rho(A)) = \{ m \in \GL(n,\Q) \st \forall_{a \in A} m \rho(a) = \rho(a) m \}.
\]

Let $m=(m_{ij}) \in \GL(n,\Q)$ be a~block matrix such that $m \rho_{|A} = \rho_{|A} m$. 
We get
\[
\begin{pmatrix}
m_{11} & \ldots & m_{1k}\\
\vdots & \ddots & \vdots\\
m_{k1} & \ldots & m_{kk}
\end{pmatrix}
\begin{pmatrix}
\rho_1 & & 0\\
       & \ddots & \\
0      & & \rho_k
\end{pmatrix}
=
\begin{pmatrix}
\rho_1 & & 0\\
       & \ddots & \\
0      & & \rho_k
\end{pmatrix}
\begin{pmatrix}
m_{11} & \ldots & m_{1k}\\
\vdots & \ddots & \vdots\\
m_{k1} & \ldots & m_{kk}
\end{pmatrix},
\]
and thus
\[
\forall_{1 \leq i,j \leq k} m_{ij} \rho_j = \rho_i m_{ij}.
\]
Since for $i \neq j$, $\rho_i$ and $\rho_j$ have no common subrepresentation, 
by Schur's Lemma (see \cite[(27.3)]{CuRe62}) $m_{ij} = 0$ for $i \neq j$ and $m_{ii} \in \GL(n/k,\Q)$, for $i=1,\ldots,k$. 
We have just proved

\begin{lm}
\label{lm:centralizer}
Let $\rho \colon G \to \GL(n,\Q)$ be a~faithful, absolutely irreducible representation of a finite group $G$ of odd degree $n$. Let $A$ be a~normal abelian subgroup of $G$ such that conditions \eqref{eq:decomposition2} and \eqref{eq:images2} hold. Then
\[
C_{\GL(n,\Q)}(\rho(A)) = \{ \diag(c_1,\ldots,c_k) \st c_i \in \GL(n/k,\Q), i=1,\ldots,k \},
\]
where $k$ is equal to the number of pairwise non-isomorphic irreducible subrepresentations of $\rho_{|A}$.
\end{lm}

\subsection{Normalizer}

\label{subsec:normalizer}

Since the group $A$ is finite, $\Aut(A)$ is a~finite group. Moreover, we have a~monomorphism
\[
N_{\GL(n,\Q)}(\rho(A))/C_{\GL(n,\Q)}(\rho(A)) \hookrightarrow \Aut(A).
\]
 
Hence any coset $mC_{\GL(n,\Q)}(\rho(A)), m\in N_{\GL(n,\Q)}(\rho(A))$
corresponds to some automorphism of $A.$

Let $\varphi \in \Aut(A)$ and $m = (m_{ij}) \in \GL(n,\Q)$ be a~block matrix, 
which represents this automorphism,
with blocks of degree $n/k,$ i.e. 
\[
\forall_{c\in C_{\GL(n,\Q)}(\rho(A))}\forall_{a \in A} (mc) \rho(a) (mc)^{-1} = m \rho(a) m^{-1} = \rho(\varphi(a)).
\] 
We have
\[
\begin{pmatrix}
m_{11} & \ldots & m_{1k}\\
\vdots & \ddots & \vdots\\
m_{k1} & \ldots & m_{kk}
\end{pmatrix}
\begin{pmatrix}
\rho_1 & & 0\\
       & \ddots & \\
0      & & \rho_k
\end{pmatrix}
=
\begin{pmatrix}
\rho_1\varphi & & 0\\
       & \ddots & \\
0      & & \rho_k\varphi
\end{pmatrix}
\begin{pmatrix}
m_{11} & \ldots & m_{1k}\\
\vdots & \ddots & \vdots\\
m_{k1} & \ldots & m_{kk}
\end{pmatrix}.
\]
Note that
\begin{equation}
\label{eq:images}
\forall_{1 \leq i \leq k} \Img(\rho_i) = \Img(\rho_i \varphi) = \langle -\I \rangle.
\end{equation}
Since, for $i \neq j$, $\rho_i$ and $\rho_j$ do not have common subrepresentations, the same applies to $\rho_i\varphi$ and $\rho_j\varphi$.
Hence, using Schur's lemma again for every $1 \leq i \leq k$ there exists exactly one $1 \leq j \leq k$ such that
\[
m_{ji} \rho_i = \rho_j \varphi m_{ji}
\]
and $m_{ji} \neq 0$. Moreover, $\det(m) \neq 0$ and also $\det(m_{ij})\neq 0$. 
By \eqref{eq:images} $\rho_i = \rho_j \varphi$ and there exists a~permutation $\sigma \in S_k$, where $S_k$ is the symmetric group on $k$ letters, such that
\begin{equation}
\label{eq:normalizer}
m \diag(\rho_1,\ldots,\rho_k) m^{-1} = \diag(\rho_{\sigma(1)}, \ldots, \rho_{\sigma(k)}).
\end{equation}

Let $\tau \in S_k$ be any permutation and let $P_\tau \in \GL(n,\Q)$ be  the block matrix, with blocks of degree $n/k$, such that
\begin{equation}
\label{eq:ptau}
(P_\tau)_{i,j} = 
\left\{
\begin{array}{ll}
 \I & \text{if } \tau(i) = j,\\
  0 & \text{otherwise},
\end{array}
\right.
\end{equation}
where $1 \leq i,j \leq k$. By \eqref{eq:normalizer} we may take
\[
m = P_\sigma
\]
as a~representative of a~coset in $N_{\GL(n,\Q)}(\rho(A))/C_{\GL(n,\Q)}(\rho(A))$, 
which realizes the automorphism $\varphi$.

Let 
\[
S := \{ \tau \in S_k \st P_\tau \in N_{\GL(n,\Q)}(\rho(A)) \}.
\]
Then $S$ is a~subgroup of $S_k$ and
\[
P := \{ P_\tau \st \tau \in S \}
\]
is a~subgroup of the normalizer. By the above and the Lemma \ref{lm:centralizer}, we get

\begin{prop}
\label{prop:normalizer}
The normalizer $N_{\GL(n,\Q)}(\rho(A))$ is a~semidirect product of  $C_{\GL(n,\Q)}(\rho(A))$ and $P$. Moreover
\[
N_{\GL(n,\Q)}(\rho(A)) = C_{\GL(n,\Q)}(\rho(A)) \cdot P \cong C_{\GL(n,\Q)}(\rho(A)) \rtimes S \cong \GL(n/k,\Q) \wr S,
\]
where $\GL(n/k,\Q) \wr S$ denotes the wreath product of $\GL(n/k,\Q)$ and $S$.
\end{prop}

\subsection{Properties of the group \texorpdfstring{$G$}{G}}

Let
\[
C:= C_{G}(A)%\rho^{-1}\bigl(\rho(G) \cap C_{\GL(n,\Q)}(\rho(A))\bigr)
\]
be the centralizer of $A$ in $G$. Since $\rho$ is faithful, we have that 
\[
C = \rho^{-1}(C_{\GL(n,\Q)}(\rho(A))).
\]

By the proposition \ref{prop:normalizer}, the kernel of the following composition
\[
G \stackrel{\rho}{\longrightarrow} N_{\GL(n,\Q)}(\rho(A)) \stackrel{\nu}{\longrightarrow} N_{\GL(n,\Q)}(\rho(A))/C_{\GL(n,\Q)}(\rho(A)) \cong S,
\]
where $\nu$ is the quotient homomorphism, equals $C$ and hence we have an isomorphism of groups
\[
S \cong G/C.
\]

The representations $\rho_i, i=1,\ldots,k$ are defined on the group $A$.  
Lemma \ref{lm:centralizer} gives us a~possibility to extend the domains of 
these representations to $C$. 
Let $V_i$ be subspaces of $\Q^n$ corresponding to 
representations $\rho_i$, $i=1,\ldots,k$. In fact, since $\rho_{|C}$ is in 
block diagonal form, we have 

\[
\forall_{1 \leq i \leq k} \; V_i = \underbrace{\Theta \oplus \ldots \oplus \Theta}_{i-1} \oplus \Q^{n/k} \oplus \Theta \oplus \ldots \oplus \Theta \subset \Q^n,
\]
where $\Theta$ is considered as a~zero-dimensional subspace (zero vector) of $\Q^{n/k}.$ 
Moreover, every element of the group $S$ permutes elements of the set 
\[
\{V_1, \ldots, V_k\}.
\]
We want to prove that this action is transitive. 
\begin{lm}
\label{lm:transitive}
$S \subset S_k$ is a~transitive permutation group.
\end{lm}
\begin{proof}
If we assume that $S$ is not transitive, then
\[
\exists_{1 \leq j \leq k} \forall_{i \neq j} \forall_{\tau \in S} \; \tau(i) \neq j.
\]
Let 
\[
\hat V_j = \mathop{\bigoplus_{i=1}^k}_{i \neq j} V_i
\]
and $g \in G$. Then $\rho(g)=P_\tau m$, for some $\tau \in S$ and $m \in \bigoplus_{i=1}^k\GL(n/k,\Q)$. We get
\[
\rho(g)(\hat V_j) = P_\tau m \cdot \hat V_j = P_\tau \cdot \hat V_j = \mathop{\bigoplus_{i=1}^k}_{i \neq j} V_{\tau(i)} = \hat V_j.
\]
Thus $\hat V_j \subsetneq \Q^n$ is an invariant subspace of $\rho$ and hence $\rho$ is reducible (over $\Q$). 
This contradiction proves the lemma.
\end{proof}

The following lemma helps us to understand the structure of the representation $\rho$.

\begin{lm}
Representations $\rho_1,\ldots,\rho_k \colon C \to \GL(n/k,\Q)$ are absolutely irreducible.
\end{lm}
\begin{proof}
Let $\phi \colon C \to \GL(d,\C)$ be a~$\C$-irreducible subrepresentation of $\rho_{|C}$. 
By Clifford's theorem, for the group $C \lhd G$ the representation $\rho_{|C}$ is a~sum of conjugates of $\phi,$ i.e.
\[
\rho_{|C} = \bigoplus_{s=1}^m \phi^{(g_s)}, 
\]
where $g_s \in G, s=1,\ldots,m$ and $g_1 = 1$. For every $1 \leq s \leq m$, $\phi^{(g_s)}$ is a~complex subrepresentation of some $\rho_i, i=1,\ldots,k$. 
Counting dimensions, we can see that for every $1 \leq i \leq k$ 
\[
\rho_i = \bigoplus_{j=1}^{m/k} \rho_{i,j},
\]
where
\[
\forall_{1 \leq j \leq m/k} \; \rho_{i,j} \in \{\phi^{(g_s)} \st 1 \leq s \leq m\}.
\]

Let $V_{i,j} \subset V_i$ be an invariant space under the action of $\rho_{i,j}$, for $1 \leq i \leq k, 1 \leq j \leq m/k$. Taking a~suitable basis for $V_i$, $1 \leq i \leq k$, we can assume that the decomposition
\[
\rho_i = \bigoplus_{j=1}^{m/k} \rho_{i,j}
\]
is given in a~block diagonal form: 
\[
\forall_{1 \leq j \leq m/k} \; V_{i,j} = \underbrace{\Theta \oplus \ldots \oplus \Theta}_{j-1} \oplus \C^{n/m} \oplus \Theta \oplus \ldots \oplus \Theta \subset V_i,
\]
where $\Theta$ is a~zero-dimensional subspace (zero vector) of $\C^{n/m}$. 
Note that the images of $\rho_i{}_{|A}, i=1\dots,k,$ remain the same in this new basis.
Hence the description of the representatives of the normalizer given in the subsection \ref{subsec:normalizer}, 
remains the same for the group $\GL(n,\C)$.

If the representations $\rho_i, i=1,\ldots,k,$ are $\C$-reducible then $m > k$. Let
\[
W = \bigoplus_{i=1}^k V_{i,1}
\]
and $g \in G$. Then $\rho(g) = P_\tau m$ (as in the proof of lemma \ref{lm:transitive}) and we get
\[
\rho(g)(W) = P_\tau m \cdot W = P_\tau \cdot W = \bigoplus_{i=1}^k V_{\tau(i),1} = W.
\]
Hence $W \subsetneq \C^n$ is an invariant subspace of $\rho$ and thus $\rho$ cannot be absolutely irreducible. This contradiction finishes the proof.
\end{proof}

\subsection{Abelian normal subgroups}

Without lost of generality we can assume that $A$ is a maximal abelian normal subgroup of $G$, i.e. if $A' \lhd G$ is abelian and $A \subset A'$ then $A = A'$. We will show, that $A$ is unique in $G$ and hence -- characteristic.

\begin{lm}
$A$ is unique in $C$.
\end{lm}
\begin{proof}
Let $A' \lhd G$ be an abelian group, such that $A' \subset C$. Since all elements of $A$ commute with all elements of $C$, they commute with all elements of $A'$. Hence $AA'$ is a normal abelian subgroup of $G$. Since $A$ is maximal, we have
\[
AA' = A \Rightarrow A' \subset A.
\]
\end{proof}

If we can prove that $A \subset C$, then $A$ is going to be unique in $G$. Recall that we have a~short exact sequence
\[
1 \longrightarrow C \longrightarrow G \stackrel{p}{\longrightarrow} S \longrightarrow 1.
\]
Assuming $A \not\subset C$, we get 
\[
1 \neq p(A) \lhd S.
\]
We prove that it is impossible.

\begin{lm}
Let $S \subset S_k$ be a~transitive permutation group and $k$ be an odd natural number. Then $S$ contains no nontrivial normal elementary abelian 2-groups.
\end{lm}
\begin{proof}
Let $x \in X = \{1,\ldots,k\}$. Let $S_x$ be the stabilizer of $x$ in $S$ and $Sx$ be the orbit of $x$. By the transitivity of the action of $S$ on $X$, we have that $Sx = X$ and since we have a~bijection
\[
Sx \leftrightarrow \{\tau S_x \st \tau \in S\},
\]
the index $[S : S_x]$ of $S_x$ in $S$ is an odd number. Now let $B$ be any normal $2$-subgroup of $S$. Then $B \subset S_x$ and we get
\[
B = \cap_{\tau \in S} \tau B \tau^{-1} \subset \cap_{\tau \in S} \tau S_x \tau^{-1} = \cap_{\tau \in S} S_{\tau(x)} = 1,
\]
since $S$ acts faithfully % and transitively
on $X$.
\end{proof}
We have just proved
\begin{prop}
The maximal, normal elementary abelian subgroup $A \lhd G$ is unique maximal in $G$ and hence it is a characteristic subgroup.
\end{prop}

\begin{cor}
\label{cor:subgroup}
\[
N_{\GL(n,\Q)}(\rho(G)) \subset N_{\GL(n,\Q)}(\rho(A)).
\]
\end{cor}

\subsection{The proof of the Theorem \ref{thma:a}}

Let us first restate the theorem.

\begin{thma}
\label{thma:a}
Let $G$ be a finite group with a non-trivial normal 
abelian subgroup $A$ and let $\rho \colon G\to \GL(n,\Z)$ 
be a faithful $\R$-irreducible representation. 
Suppose $n$ is odd. Then for every $D\in N_{\GL(n,\Z)}(\rho(G)),$ there exists $g\in G$ such that $\rho(g) D$ has eigenvalue $1.$
\end{thma}

\begin{proof}
%Note first, that eigenvalues of matrices and their products does not depend on their conjugacy class. 
%Hence, we can change the basis of $\rho$, with conjugating the group $N_{\GL(n,\Z)}(\rho(G))$ by appropriate invertible rational matrix simultaneously, 
%and prove the theorem with these new forms of $\rho$ and $N = N_{\GL(n,\Z)}(\rho(G)).$ 
Note that, by $\R$-irreducibility of $\rho$, $N = N_{\GL(n,\Z)}(\rho(G))$ is a~finite group (see \cite[pages 587-588]{S}).

Since eigenvalues of matrices % and their products 
do not depend on their conjugacy class, we can assume that $\rho(A)$ is a~group of diagonal matrices. 
Using Corollary \ref{cor:subgroup}, Proposition \ref{prop:normalizer} and the fact that
\[
N \subset N_{\GL(n,\Q)}(\rho(G)),
\]
we get
\[
N \subset C_{\GL(n,\Q)}(\rho(A)) \cdot P.
\]
Recall that
\[
C_{\GL(n,\Q)}(\rho(A)) = \bigoplus_{i=1}^k \GL(n/k,\Q)
\]
and elements of $P$ are ''block permutation matrices'' (see Lemma \ref{lm:centralizer} and \eqref{eq:ptau} respectively).

Let $D \in N,$ then $D$ has the form
\[
D = P_\sigma \diag(c_1,\ldots,c_k),
\]
where $\sigma \in S_k$ and $c_i \in \GL(n/k,\Q)$, for $i=1,\ldots,k$. Recall that $G/C \cong S$, 
where $S \subset S_k$ is a~transitive permutation group (see Lemma \ref{lm:transitive}). Hence there exists $\tau \in S$ such that 
\[
\tau(1) = \sigma^{-1}(1)
\]
and for some $g' \in G$,
\[
\rho(g') = \diag(c_1',\ldots,c_k') P_\tau.
\]
We get
\[
\begin{split}
\rho(g')D & = \diag(c_1',\ldots,c_k') P_\tau P_\sigma \diag(c_1,\ldots,c_k) \\
 & = \diag(c_1',\ldots,c_k') P_{\sigma\tau} \diag(c_1,\ldots,c_k) = \diag(d,T),
\end{split}
\]
where $T$ is the matrix of rows of $\diag(c_2,\ldots,c_k)$ permuted by $\sigma\tau$. 
Since $d=c_1'c_1 \in \GL(n/k,\Q)$ has an odd degree, it must have real eigenvalue 
and since $N$ is of a~finite order, this eigenvalue is $\pm 1$. 
If the eigenvalue is $1$, then we take $g=g'$ and the theorem is proved. Otherwise, by the Clifford's theorem and the faithfulness of $\rho$, we can take $a \in A$ such that $\rho_1(a) = -\I$. Then $\rho_1(a)d$ has an eigenvalue $1$ and hence, taking $g=ag'$, the element
\[
\begin{split}
\rho(g)D & = \rho(ag')D = \rho(a) \rho(g')D = \rho(a) \diag(d,T) = \\ 
         & = (\rho_1\oplus\ldots\oplus\rho_k)(a)\cdot \diag(d,T) = \\ &=\diag(\rho_1(a)d,(\rho_2\oplus\ldots\oplus\rho_k)(a)T)
\end{split}
\]
has an eigenvalue equal to $1$ also. This finishes the proof.
\end{proof}

\end{document}